\documentclass[a4paper,10pt]{article}

\usepackage{amsmath}
\usepackage{amssymb}
\usepackage{amsthm}
\usepackage{amsrefs}
\usepackage{color}

\newcommand{\NN}{{\mathbb{N}}}
\newcommand{\NNI}{{\NN}_{0}}
\newcommand{\ZZ}{{\mathbb{Z}}}
\newcommand{\RR}{{\mathbb{R}}}
\newcommand{\RRP}{{\RR}_{>0}}

\newcommand{\bs}[1]{{\boldsymbol{#1}}}
\newcommand{\bsa}{\bs{\alpha}}
\newcommand{\bsb}{\bs{\beta}}
\newcommand{\bslm}{\bs{\lambda}}
\newcommand{\bsn}{\bs{\nu}}

\newcommand{\bsr}{\bs{\rho}}

\newcommand{\wbs}[1]{{\widetilde{\bs{#1}}}}
\newcommand{\wbsa}{\wbs{\alpha}}
\newcommand{\wbsb}{\wbs{\beta}}
\newcommand{\wbsl}{\wbs{\lambda}}
\newcommand{\wbsn}{\wbs{\nu}}

\newcommand{\wR}{\widetilde{R}}
\newcommand{\wV}{\widetilde{V}}

\newcommand{\iprd}[2]{{\langle{#1},{#2}\rangle}}
\newcommand{\inrm}[1]{{\lvert{#1}\rvert}}

\newcommand{\iiprd}[2]{{\langle\!\langle{#1},{#2}\rangle\!\rangle}}
\newcommand{\iinrm}[1]{{\lVert{#1}\rVert}}

\newcommand{\AVV}{A(V;V)}
\newcommand{\AVR}{A(V; \RR )}

\newcommand{\map}{\mathop{\mathrm{Map}}\nolimits}
\newcommand{\mult}{\mathop{\mathrm{mult}}\nolimits}

\newcommand{\GL}{{\mathrm{GL}}}

\theoremstyle{plain}
\newtheorem{theo}{Theorem}[section]
\newtheorem{defi}[theo]{Definition}
\newtheorem{prop}[theo]{Proposition}
\newtheorem{lemm}[theo]{Lemma}
\newtheorem{coro}[theo]{Corollary}
\newtheorem{rema}[theo]{Remark}

\begin{document}

\title{A characterization of root systems \\ from the viewpoint of denominator formulae}
\author{%
Hiroki Aoki%
\thanks{aoki\_hiroki\_math@nifty.com,
Faculty of Science and Technology,
Tokyo University of Science.}
\and
Hiraku Kawanoue%
\thanks{kawanoue@fsc.chubu.ac.jp,
College of Science and Engineering,
Chubu University.}
}

\maketitle

\begin{abstract}

Root systems are sets with remarkable symmetries and
therefore they appear in many situations in mathematics.
Among others, denominator formulae of root systems are
very beautiful and mysterious equations which have
several meanings from a variety of disciplines in mathematics.
In this paper, we show a converse statement of this phenomena.
Namely,
for a given finite subset \( S \) of a Euclidean vector space \( V \),
define an equation \( F \) in the group ring \( \ZZ[V] \)
featuring the product part of denominator formulae.
Then, a geometric condition for the support
of \( F \) characterizes \( S \) being a set of positive roots of
a finite/affine root system, recovering the denominator formula.
This gives a novel characterization of the sets of positive roots of
reduced finite/affine root systems.

\end{abstract}


\bigskip

\noindent
{\underline{\bf{Notation}}}

\smallskip

\noindent
We denote
by \( \NN \), \( \NNI \), \( \ZZ \), \( \RRP \) and \( \RR \) the
sets of all positive integers, all non-negative integers, all integers,
all positive real numbers and all real numbers, respectively.
For sets \( A , A_1 \) and \( A_2 \), \( A= A_1 \sqcup A_2 \) means
that \( A \) is the disjoint union of \( A_1 \) and \( A_2 \),
namely, both \( A= A_1 \cup A_2 \) and \( A_1 \cap A_2 =
\emptyset \) hold.

\section{Introduction}

Root systems were originally introduced by Killing around 1889,
in order to classify simple Lie algebras over the field of complex numbers.
They are not only fundamental in the theory of Lie groups and Lie algebras,
but also play important roles in many areas in mathematics such as
the singularity theory or the graph theory.

Root systems have several important properties.
Among others,
the Weyl denominator formulae for finite root systems
or
the Weyl-Kac denominator formulae, also known as the Macdonald identities,
for affine root systems are
very beautiful and mysterious equations.
They have several meanings according to the context of various areas
in mathematics, and afford many attractive equations such as
the Jacobi triple product identity.
Let us look at the Weyl denominator formula
of a finite root system \( R \)
\[
 \prod_{\bsa \in R^+} \left( 1- e^{\bsa} \right)
 =
 \sum_{w \in W(R)} ( \det w) e^{ \bsr -w( \bsr )}
\]
more closely.
The left hand side of this formula is
the product of the factors \( ( 1- e^{\bsa} ) \) with \( \bsa \) running
over the elements of the set of positive roots \( R^+ \).
The right hand side is its expansion,
where many terms are cancelling out and surviving terms are
characterized in terms of the Weyl vector \( \bsr \).
Actually, as we see later, the set of exponents in the right hand side
\( \{ \; \bsr -w( \bsr ) \; | \; w \in W(R) \; \} \)
lies on a sphere.

Now relax the condition for the left hand side by \( R^+ \) to be
a finite subset \( S \) of a Euclidean vector space, by allowing to have
positive exponents \( (1- e^{\bs{s}} )^{m( {\bs{s}} )} \) for
each \( {\bs{s}} \in S \), and obtain the equation
\[
\prod_{ {\bs{s}} \in S} \left( 1- e^{\bs{s}} \right)^{m( {\bs{s}})}
= \sum_{\bslm} c_{\bslm} e^{\bslm}.
\]
A natural question is how one can characterize
the denominator formulae among these general equations.
As mentioned above, one necessary condition is that
the set of exponents of the support of the right hand side
\( \{ \; \bslm \; | \; c_{\bslm} \neq 0 \; \} \) lies on a sphere.
Remarkably, it is already a sufficient condition.
Namely, this condition implies that \( m( {\bs{s}} )=1 \) for \(
{\bs{s}} \in S \) and that \( S \) is
the half of the set of all roots of a finite root system.
In other words, it gives a characterization of
the set of positive roots of a finite root system
in terms of the expansion of the equation
featuring the product part of denominator formulae.

The set of positive roots of an affine root system
also possesses a similar characterization.
The condition is almost same to that in the case of a finite root system,
but the set of exponents lies on a paraboloid, instead of a sphere.

\bigskip

This paper consists of three sections.
\S1 is an introduction. The Weyl denominator formula is
discussed to convey the motivation and idea of our work.
\S2 is devoted to the case of finite root systems.
After reviewing the definition of finite root systems,
the statement and the proof of their characterization are given.
\S3 is devoted to the case of affine root systems.
The definition of affine root systems after Macdonald
is given, and their characterization is presented.

\section{The case corresponding to finite root systems}

In this section, we give a characterization of
the sets of positive roots of finite root systems.

\subsection{Finite root systems}

Let \( N \in \NN \) and
let \( V \) be an \( N \)-dimensional real Euclidean space.
We denote
by \( \iprd{\cdot}{\cdot} \) the
standard Euclidean inner product of \( V \) and
by \( \inrm{\cdot} \) the
standard Euclidean norm of \( V \).
For \( \bsa \in V \setminus \{ \bs{0} \} \),
the reflection \( w_{\bsa}  : V \to V \) is
defined by
\[
 w_{\bsa} ( \bs{v} )
 := \bs{v} -
 \frac{2 \iprd{\bsa}{\bs{v}}}{\inrm{\bsa}^2}
 \bsa
 ,
\]
which is an invertible linear transformation
on \( V \), namely \( w_{\bsa} \in \GL (V) \).
Since \( w_{\bsa} ( \bsa )=- \bsa \) and \(
w_{\bsa} ( \bsb )= \bsb \) if \( \iprd{\bsa}{\bsb} =0 \),
we have \( \det w_{\bsa} =-1 \).
Since \( \inrm{w_{\bsa} ( \bs{v} )} = \inrm{\bs{v}} \) for
any \( \bs{v} \in V \), \( w_{\bsa} \) is
an orthogonal transformation of \( V \).
For \( R \subset V \setminus \{ \bs{0} \} \),
we denote by \( W(R) \) the subgroup of \( \GL (V) \) generated by
reflections \( \{ \; w_{\bsa} \; | \; \bsa \in R \; \} \).
\begin{defi}
We say \( R \subset V \setminus \{ \bs{0} \} \) is
a finite root system of rank \( N \) if
\( R \) satisfies the following three conditions:
\begin{description}
\item[(FR1)]
\( R \) spans \( V \) as an \( \RR \)-vector space.
\item[(FR2)]
For any \( \bsa \in R \), \( w_{\bsa} (R)=R \).
\item[(FR3)]
For any \( \bsa , \bsb \in R \),
\( \dfrac{2 \iprd{\bsa}{\bsb}}{\inrm{\bsa}^2} \in \ZZ \).
\end{description}
\end{defi}
\begin{prop}
\label{prop:Ffnt}
If \( R \) is a finite root system,
then the following property holds.
\begin{description}
\item[(FR4)]
\( R \) is finite but nonempty.
\end{description}
Hence \( W(R) \) is a finite group.
\end{prop}
\begin{proof}
We take \( \bsb_1 , \dots , \bsb_N \in R \) which form
a basis of \( V \) as an \( \RR \)-vector space.
Since
\[
 \left| \dfrac{2 \iprd{\bsa}{\bsb}}{\inrm{\bsa}^2} \right|
 \left| \dfrac{2 \iprd{\bsa}{\bsb}}{\inrm{\bsb}^2} \right|
 \leqq 4
\]
holds for any \( \bsa , \bsb \in R \) by
the Cauchy-Schwarz inequality, {\bf{(FR3)}} implies
\[
 \dfrac{2 \iprd{\bsa}{\bsb}}{\inrm{\bsb}^2} \in
 \ZZ \cap [-4,4]
 .
\]
Thus we have
\[
 R \subset
 \left\{ \: \bsa \in V \; \middle| \;
 \dfrac{2 \iprd{\bsa}{\bsb_j}}{\inrm{\bsb_j}^2} \in
 \ZZ \cap  [-4,4]
 \quad \text{for any} \quad j \in \{ \; 1, \dots ,N \; \} \; \right\}
 .
\]
Because the right hand side is a finite set,
the condition {\bf{(FR4)}} holds.
\end{proof}
We remark that
there are several styles to define root systems.
The standard definition includes the condition {\bf{(FR4)}},
although it follows
from {\bf{(FR1)}} and {\bf{(FR3)}}. (cf.~\cites{Ha1,Hu1})

Here we prepare one elementary lemma,
which is used repeatedly in this section.
\begin{lemm}
\label{lemm:Ftkn}
Let \( S \) be a nonempty subset of \( V \) which is
at most countable and let \( \bs{p} \in V \).
Then there exists a nonzero vector \( \bsn \in V \) such that
\[
 \{ \; \bs{s} \in S \; | \; \iprd{\bs{s}}{\bsn} =0 \; \}
 =
 \RR \bs{p} \cap S
 .
\]
\end{lemm}
\begin{proof}
For \( \bs{t} \in V \), we denote by \( V_{\bs{t}} :=
\{ \; \bs{v} \in V \; | \; \iprd{\bs{t}}{\bs{v}} =0 \; \}
\) the orthogonal complement of \( \RR \bs{t} \).
For each \( \bs{s} \in S \setminus \RR \bs{p} \), \(
V_{\bs{p}} \cap V_{\bs{s}} \) is
a subspace of \( V_{\bs{p}} \) with codimension \( 1 \).
Since \( S \) is at most countable, \(
\bigcup_{ \bs{s} \in S \setminus \RR \bs{p} }
( V_{\bs{p}} \cap V_{\bs{s}} ) \) is
a proper subset of \( V_{\bs{p}} \).
Hence we can take
a nonzero vector \( \bsn \in V_{\bs{p}} \) from its complement.
\end{proof}
Let \( R \) be a finite root system.
We take a nonzero vector \( \bsn \in V \) according
to {\bf{Lemma \ref{lemm:Ftkn}}} with \( S=R \) and \( \bs{p}
= \bs{0} \).
We denote by
\[
 R^+ := \{ \; \bsa \in R \; | \; \iprd{\bsa}{\bsn} >0 \; \}
\]
the set of all positive roots.
We remark that \( R= R^+ \sqcup (- R^+ ) \) holds,
where \( - R^+ := \{ \; - \bsa \; | \; \bsa \in R^+ \; \} \).
Let
\[
 B(R):= R^+ \setminus
 \{ \; \bsa \in R^+ \; | \; \exists \bsa_1 , \bsa_2 \in R^+
 \; \mathrm{s.t.} \; \bsa = \bsa_1 + \bsa_2 \; \}
\]
be a basis for \( R \).
This \( B(R) \) is uniquely determined by \( R^+ \) but
depends on the choice of \( R^+ \).
\begin{defi}
We say a finite root system \( R \) is reduced
if \( R \) satisfies the following condition:
\begin{description}
\item[(FR5)]
For any \( \bsa \in R \), \( \RR \bsa \cap R = \{ \bsa , - \bsa \} \).
\end{description}
\end{defi}
Let \( R \) be a reduced finite root system.
For \( w \in W(R) \), let
\[
 R(w) := \{ \; \bsa \in R^+ \; | \; w( \bsa ) \not\in R^+ \; \}
 \quad \text{and} \quad
 s(w) := \sum_{\bsa \in R(w)} \bsa
 ,
\]
where we define \( s(w)= \bs{0} \) if \( R(w)= \emptyset \).
Then, the following equation, which is called
the Weyl denominator formula, holds:
\[
 \prod_{\bsa \in R^+} \left( 1- e^{\bsa} \right)
 =
 \sum_{w \in W(R)} ( \det w) e^{s(w)}
 .
\]
This is a formal equation
with multiplication \( e^{\bs{v}_1} \cdot e^{\bs{v}_2} =
e^{\bs{v}_1 + \bs{v}_2} \),
where \( e^{\bs{v}} \) is a formal symbol for each \( \bs{v} \in V \) and
we identify \( e^{\bs{0}}=1 \).
It is easy to see that \( s(w)= \bsr - w( \bsr ) \),
where
\[
 \bsr := \frac{1}{2} \sum_{\bsa \in R^+} \bsa
\]
is the Weyl vector.

\subsection{Statement for the finite case}

Let \( m \in \map (V, \NNI ) \), namely,
\( m \) is a map from \( V \) to \( \NNI \).
Here we assume that \( m \) satisfies
the following properties {\bf{(F)}}:
\[
 \text{\bf{(F)}} \quad
 \left\{
 \begin{array}{l}
 \bullet \;
 m( \bs{0} )=0.
 \\[4pt]
 \bullet \;
 \text{The support \( S(m):= 
 \{ \; \bs{v} \in V \; | \; m( \bs{v} ) \neq 0 \; \} 
 \) is a nonempty finite set.}
 \\[4pt]
 \bullet \;
 \text{\( S(m) \) spans \( V \) as an \( \RR \)-vector space.}
 \end{array}
 \right.
 \hspace*{-8pt}
\]
We define
\[
 F(m) := \prod_{\bs{s} \in S(m)}
 \left( 1- e^{\bs{s}} \right)^{m( \bs{s} )}
\]
and let \( \Lambda (m) \) be the set of vectors which appear in
the expansion of \( F(m) \), namely,
\[
 \Lambda (m) := \left\{ \; \bslm \in V \; \middle| \;
 c_{\bslm} \neq 0 \; \right\}
 \qquad \left(
 F(m)= \sum_{\bslm} c_{\bslm} e^{\bslm}
 \right)
 .
\]
Our assertion in this section is as follows.
\begin{theo}
\label{theo:Faim}
Under the assumption {\bf{(F)}},
the following two conditions on \( m \) are equivalent:
\begin{description}
\item[(F1)]
\( \Lambda (m) \) is a subset of a sphere.
Namely, there exist
a centre \( \bs{c} \in V \) and a radius \( r>0 \) such that \(
\inrm{\bslm - \bs{c}}=r \) holds for
any \( \bslm \in \Lambda (m) \).
\item[(F2)]
\( S(m) \cap (-S(m)) = \emptyset \),
\( R:= S(m) \sqcup (-S(m)) \) is
a reduced finite root system of rank \( N \) and \(
m( \bs{s} )=1 \) for any \( \bs{s} \in S(m) \).
\end{description}
\end{theo}
Here we prepare one lemma for later use.
\begin{lemm}
\label{lemm:Fr2s}
Let \( \bsb \in S(m) \) and \( m^{\prime} \) be
a map which is defined by
\[
 m^{\prime} ( \bs{v} ):=
 \left\{ \begin{array}{ll} m( \bs{v} )-1 & ( \bs{v} = \bsb ) \\
 m( \bs{v} )+1 & ( \bs{v} =- \bsb ) \\
 m( \bs{v} ) & (\text{otherwise}) \end{array} \right.
 .
\]
Then we have \( \Lambda ( m^{\prime} ) = \Lambda (m) - \bsb
:= \left\{ \; \bslm - \bsb \; \middle| \;
\bslm \in \Lambda (m) \; \right\} \).
Thus the condition {\bf{(F1)}} holds for \( m^{\prime} \) if
and only if the condition {\bf{(F1)}} holds for \( m \).
\end{lemm}
\begin{proof}
Since
\[
 \left( 1- e^{\bsb} \right) F ( m^{\prime} )
 =
 \left( 1- e^{- \bsb} \right) F(m)
 ,
\]
we have
\[
 F ( m^{\prime} )= - e^{- \bsb} F(m)
\]
and therefore
\( \Lambda ( m^{\prime} ) = \Lambda (m) - \bsb \).
\end{proof}
\subsection{Proof for the finite case}

The proof for {\bf{(F2)}}\( \Rightarrow \){\bf{(F1)}} is easy.
Let \( R \) be a reduced finite root system of rank \( N \) and
let \( S \) be a subset of \( R \) satisfying \( R=S \sqcup (-S) \).
Let \( m_S \) be a map defined by
\[
 m_S ( \bs{v} ) := \left\{ \begin{array}{lc} 1 & ( \bs{v} \in S) \\
 0 & ( \bs{v} \not\in S) \end{array} \right.
 .
\]
What we need to do here is to show
that the condition {\bf{(F1)}} holds for this \( m_S \).
By {\bf{Lemma \ref{lemm:Fr2s}}},
we may assume \( S= R^+ \).
Then, by the Weyl denominator formula, we have
\[
 \Lambda ( m_S ) \subset
 \left\{ \; \bsr - w( \bsr ) \; \middle| \; w \in W(R) \; \right\}
 .
\]
Since \( w \) is an isometry,
we have \( \inrm{w ( \bsr )} = \inrm{\bsr} \).
This means that \( \Lambda ( m_S ) \) is a subset of the sphere
with its centre \( \bsr \) and its radius \( \inrm{\bsr} \).
This completes the proof of {\bf{(F2)}}\( \Rightarrow \){\bf{(F1)}}.

\bigskip

We show {\bf{(F1)}}\( \Rightarrow \){\bf{(F2)}} step by step.
Throughout the proof, the key is the fact that the intersection
of a sphere and a line has at most two points.
First, we show {\bf{(FR5)}}.
\begin{lemm}
\label{lemm:Fred}
Assume {\bf{(F1)}} and let \( \bsa \in S(m) \).
Then we have \( \RR \bsa \cap S(m) = \{ \bsa \} \) and \(
m( \bsa )=1 \).
\end{lemm}
\begin{proof}
We take a nonzero vector \( \bsn \in V \) according
to {\bf{Lemma \ref{lemm:Ftkn}}} with \( S=S(m) \) and \(
\bs{p} = \bsa \).
By {\bf{Lemma \ref{lemm:Fr2s}}},
we may assume \( \iprd{\bs{s}}{\bsn} \geqq 0 \) for
any \( \bs{s} \in S(m) \) without
changing \( \RR \bsa \cap S(m) \).
Then, because
\[
 S_0 :=
 \{ \; \bs{s} \in S(m) \; | \; \iprd{\bs{s}}{\bsn} =0 \; \}
 = \RR \bsa \cap S(m)
 ,
\]
we have
\[
 \Lambda_0 :=
 \{ \; \bslm \in \Lambda (m) \; | \; \iprd{\bslm}{\bsn} =0 \; \}
 = \Lambda ( m |_{\RR \bsa} )
 ,
\]
where \( m |_{\RR \bsa} \in \map ( \RR \bsa , \NNI )\) is
the restriction of \( m \) to
the subspace \( \RR \bsa \subset V \).
Hence \( \Lambda_0 \) is on the line \( \RR \bsa \).
On the other hand,
since \( \Lambda_0 \subset \Lambda (m) \), \( \Lambda_0 \) is
on a sphere, by the assumption {\bf{(F1)}}.
Therefore, \( \Lambda_0 \) is on the intersection
of the line and the sphere, which has at most two points.
However, since \( \bsa \in S_0 \), \( \Lambda_0 \) has
at least two points.
Consequently, \( \Lambda_0 \) has
exactly two points, which should be \( \bs{0} \) and \( \bsa \).
This means \( S_0 = \RR \bsa \cap S(m) = \{ \bsa \} \) and \(
m( \bsa )=1 \).
\end{proof}
Next, we show the {\bf{(FR3)}}.
\begin{lemm}
\label{lemm:Fint}
Assume {\bf{(F1)}} and let \( \bsa \in S(m) \).
Then we have \(
\frac{2 \iprd{\bsa}{\bsb}}{\inrm{\bsa}^2}
\in \ZZ \) for any \( \bsb \in S(m) \).
\end{lemm}
\begin{proof}
We take a nonzero vector \( \bsn \in V \)
as in the previous proof.
Then we may assume \( \iprd{\bsa}{\bsn} =0 \) and \(
\iprd{\bs{s}}{\bsn} > 0 \) for
any \( \bs{s} \in S(m) \setminus \{ \bsa \} \).
According to the proof of {\bf{Lemma \ref{lemm:Fred}}},
we have \( \{ \bs{0} , \bsa \} \subset \Lambda (m) \).
Therefore, \( \bs{c} \in V \),
the centre of the sphere in the assumption {\bf{(F1)}},
satisfies the condition \( \inrm{\bs{0} - \bs{c}}
= \inrm{\bsa - \bs{c}} \).
Hence we have \( 2 \iprd{\bsa}{\bs{c}}
= \iprd{\bsa}{\bsa} \).

We show the assertion by contradiction.
Among the vectors \( \bsb \in S(m) \) satisfying \(
\frac{2 \iprd{\bsa}{\bsb}}{\inrm{\bsa}^2} \not\in \ZZ \),
we take one with the minimal value of \( \iprd{\bsb}{\bsn} \).
Obviously \( \bsb \neq \bsa \),
hence \( \iprd{\bsb}{\bsn} >0 \).
Here, we may assume \( \bsb -k \bsa \not\in S(m) \) for any \( k>0 \),
because \( S(m) \) is finite.
Let \( L:= \bsb + \RR \bsa \),
which is a line and a subset of \( \{ \; \bs{l} \in V \; | \;
\iprd{\bs{l}}{\bsn} = \iprd{\bsb}{\bsn} \; \} \).
To show the contradiction,
we investigate the set \( L \cap \Lambda (m) \) precisely.
Now we put
\begin{align*}
 S_1 & :=
 \{ \; \bs{s} \in S(m) \; | \;
 0< \iprd{\bs{s}}{\bsn} < \iprd{\bsb}{\bsn} \; \}
 ,
 \\
 S_2 & :=
 \{ \; \bs{s} \in S(m) \; | \;
 \iprd{\bs{s}}{\bsn} = \iprd{\bsb}{\bsn} \; \}
 = S(m) \cap L
\end{align*}
and
\[
 \mathcal{T}
 :=
 \left\{ \; T \subset S_1 \; \middle| \;
 \# T \geqq 2, \; \sigma (T):= \sum_{\bs{s} \in T} \bs{s} \in L \; \right\}
 .
\]
Then we have
\begin{equation}
\label{eq:Fcls}
 \underbrace{
 \left( 1- e^{\bsa} \right)
 \left(
 \sum_{T \in \mathcal{T}}
 (-1)^{\# T} e^{\sigma (T)}
 \right)
 }_{\textrm{(1.1)}}
 -
 \underbrace{
 \left( 1- e^{\bsa} \right)
 \left( \sum_{\bs{s} \in S_2 } e^{\bs{s}} \right)
 }_{\textrm{(1.2)}}
 =
 \sum_{\bslm \in L \cap \Lambda (m)} c_{\bslm} e^{\bslm}
 .
\end{equation}
Because \( \frac{2 \iprd{\bsa}{\bs{s}}}{\inrm{\bsa}^2} \in \ZZ \) for
any \( \bs{s} \in S_1 \),
the term \( e^{\bsb} \) cannot appear in (1.1),
while it appears in (1.2).
Therefore we obtain \( \bsb \in L \cap \Lambda (m) \) and hence \(
L \cap \Lambda (m) \) has at least two points.
On the other hand, since \( L \) is a line and
since \( \Lambda (m) \) is on a sphere, \( L \cap \Lambda (m) \) has
at most two points.
Hence \( L \cap \Lambda (m) \) has
exactly two points \( \bsb \) and \( \bsb +k \bsa \),
where \( k>0 \).
Because this \( \bsb +k \bsa \) also appears in (1.2),
we have \( k \in \ZZ \).
Since both \( \bsb \) and \( \bsb +k \bsa \) are on \( \Lambda (m) \),
we have \( \inrm{\bsb - \bs{c}} = \inrm{\bsb +k \bsa - \bs{c}} \),
therefore we have \( 2 \iprd{\bsa}{ \bsb - \bs{c}} =-k \iprd{\bsa}{\bsa} \).
Then, by \( 2 \iprd{\bsa}{\bs{c}} = \iprd{\bsa}{\bsa} \),
we have
\[
 \frac{2 \iprd{\bsa}{\bsb}}{\inrm{\bsa}^2} =-k+1 \in \ZZ
 ,
\]
which is a contradiction.
\end{proof}
Finally, we show {\bf{(FR2)}}.
\begin{lemm}
\label{lemm:Fcls}
Assume {\bf{(F1)}} and let \( \bsa \in S(m) \).
Then we have \( w_{\bsa} ( \bsb ) \in S(m) \sqcup (-S(m)) \) for
any \( \bsb \in S(m) \).
\end{lemm}
\begin{proof}
As in the previous proof,
we may assume \( \iprd{\bsa}{\bsn} =0 \) and \(
\iprd{\bs{s}}{\bsn} > 0 \) for
any \( \bs{s} \in S(m) \setminus \{ \bsa \} \).
Here we show that the reflection \( w_{\bsa} \) is a bijection
from \( S(m) \setminus \{ \bsa \} \) to itself by contradiction.
Among the vectors \( \bsb \in S(m) \setminus \{ \bsa \} \) satisfying \(
w_{\bsa} ( \bsb ) \not\in S(m) \setminus \{ \bsa \} \),
we take one with the minimal value of \( \iprd{\bsb}{\bsn} \).
Let \( L:= \bsb + \RR \bsa \),
which is a line and a subset of \( \{ \; \bs{l} \in V \; | \;
\iprd{\bs{l}}{\bsn} = \iprd{\bsb}{\bsn} \; \} \).
To show the contradiction,
recall the formula (\ref{eq:Fcls}) and
see its right hand side precisely.
Obviously, \( L \cap \Lambda (m) \) has at most two points.
If \( L \cap \Lambda (m) \) has
exactly two points \( \bslm_1 \) and \( \bslm_2 = \bslm_1 +k \bsa \),
we have
\[
 \frac{2 \iprd{\bsa}{\bslm_1}}{\inrm{\bsa}^2} =-k+1
\]
as in the previous proof, and hence we have \(
\bslm_2 = w_{\bsa} ( \bslm_1 ) + \bsa \).
Therefore, in our current setup,
the formula (\ref{eq:Fcls}) can be rewritten as
\[
 \left( 1- e^{\bsa} \right)
 \left(
 \sum_{T \in \mathcal{T}}
 (-1)^{\# T} e^{\sigma (T)}
 \right)
 -
 \left( 1- e^{\bsa} \right)
 \left( \sum_{\bs{s} \in S_2 } e^{\bs{s}} \right)
 =
 c_{\bslm_1} e^{\bslm_1} + c_{\bslm_2} e^{w_{\bsa}(\bslm_1)+\bsa}
 ,
\]
where \( c_{\bslm_1} \) and \( c_{\bslm_2} \) may vanish
if \( L \cap \Lambda (m) \) has less than two points.
Since the right hand side of this equation is a multiple
of \( \left( 1- e^{\bsa} \right) \),
we have \( c_{\bslm_2} =- c_{\bslm_1} \).
Consequently, the equation (\ref{eq:Fcls}) turns to be
\[
 \left( 1- e^{\bsa} \right)
 \left(
 \sum_{T \in \mathcal{T}}
 (-1)^{\# T} e^{\sigma (T)}
 \right)
 -
 \left( 1- e^{\bsa} \right)
 \left( \sum_{\bs{s} \in S_2 } e^{\bs{s}} \right)
 =
 c_{\bslm_1} \left( e^{\bslm_1} - e^{w_{\bsa}(\bslm_1)+\bsa} \right)
 .
\]
Namely, we have
\[
 \sum_{\bs{s} \in S_2 } e^{\bs{s}}
 =
 \left( \sum_{T \in \mathcal{T}}
 (-1)^{\# T} e^{\sigma (T)} \right)
 - c_{\bslm_1} \left(
 \frac{e^{\bslm_1} - e^{w_{\bsa}(\bslm_1)+\bsa}}{1- e^{\bsa}} \right)
 .
\]
Here we remark that \( S_1 \) is
invariant under the action of \( w_{\bsa} \) by
the choice of \( \bsb \) and that \( L \) is
invariant under the action of \( w_{\bsa} \) because \(
w_{\bsa} ( \bs{v} ) - \bs{v} \in \RR \bsa \).
Hence, in the above equation,
the expression in the first brackets on the right hand side
is also invariant under the action of \( w_{\bsa} \).
We can see that the expression in the second brackets on the right hand side
is invariant under the action of \( w_{\bsa} \) easily.
Hence \( S_2 \) is
invariant under the action of \( w_{\bsa} \),
which is a contradiction.
\end{proof}
The condition {\bf{(FR1)}} is contained in the assumption {\bf{(F)}} and
now we completes the proof of {\bf{(F1)}}\( \Rightarrow \){\bf{(F2)}}.

\bigskip

\begin{rema}
\label{rema:Finf}
We remark that
{\bf{Theorem \ref{theo:Faim}}} is not correct
if \( S(m) \) could be infinite.
To give a counter example, first we consider
the following equation on \( \RR [[X]] \) (
cf.~\cite{St1}*{Chapter 1, Exercise 40}):
\[
 \prod_{k \in \NN} \left( 1- X^k \right)^{a_k} = 1-2X
 .
\]
According to the definition of formal power series,
one can see inductively that \( a_k \in \ZZ \) for each \( k \in \NN \).
On the other hand,
by taking the logarithm of the both sides of above equation,
we have
\[
 a_k = \frac{1}{k} \sum_{d|k} \mu \left( \frac{k}{d} \right) 2^d
 ,
\]
where \( \mu \) is the M{\"o}bius function.
Thus, we have
\[
 a_k \geqq \frac{1}{k} \left( 2^k - \sum_{d=1}^{k-1} 2^d \right)
 = \frac{2}{k} >0
\]
and hence \( a_k \in \NN \) for each \( k \in \NN \).
Consequently, the following \( m \) is a counterexample
when \( N=1 \).
\[
 m (k):=
 \left\{ \begin{array}{lc} a_k & ( k \in \NN ) \\
 0 & (\text{otherwise}) \end{array} \right.
 .
\]

\end{rema}
\begin{rema}
We remark that
{\bf{Theorem \ref{theo:Faim}}} is not correct
if \( m \) could take negative integers.
Consider the case \( N=2 \), let \( \bsa := (1, \sqrt{3}), \;
\bsb :=(1, - \sqrt{3} ) \) and
\[
 m ( \bs{v} ):=
 \left\{ \begin{array}{lc} 1 & ( \bs{v} = 2 \bsa , 2 \bsb , \bsa + \bsb ) \\
 -1 & ( \bs{v} = \bsa , \bsb ) \\ 0 & (\text{otherwise}) \end{array} \right.
 .
\]
Obviously, it is not from a root system.
However, we have
\begin{align*}
 \frac{\left( 1- e^{2 \bsa} \right) \left( 1- e^{2 \bsb} \right)
 \left( 1- e^{\bsa + \bsb} \right)}{\left( 1- e^{\bsa} \right)
 \left( 1- e^{\bsb} \right)}
 & =
 \left( 1+ e^{\bsa} \right) \left( 1+ e^{\bsb} \right)
 \left( 1- e^{\bsa + \bsb} \right)
 \\
 & \hspace*{-24pt} =
 1+ e^{\bsa} + e^{\bsb} - e^{2 \bsa + \bsb} - e^{\bsa +2 \bsb}
 - e^{2 \bsa +2 \bsb}
\end{align*}
and therefore \( S(m) \) is on the sphere of
its centre \( (2,0) \) and its radius \( 2 \).
\end{rema}
%

\section{The case corresponding to affine root systems}

In this section, we give a characterization of
the sets of positive roots of affine root systems.

\subsection{Affine root systems}

Let \( N \in \NN \) and
let \( V \) be an \( N \)-dimensional real Euclidean space.
Let \( \wV := \RR \oplus V \) be
an \( (N+1) \)-dimensional real linear space.
We denote by \( p_1 : \wV \to \RR \) and \( p_2 : \wV \to V \) be
projections to the first and second components.
Namely, for \( \wbs{v} =(n, \bs{v} ) \in \wV \), we define \(
p_1 ( \wbs{v} ):=n \) and \(p_2 ( \wbs{v} ):= \bs{v} \).
We denote by \( \iprd{\cdot}{\cdot} \) the standard
Euclidean inner product of \( \wV \) and
by \( \inrm{\cdot} \) the standard
Euclidean norm of \( \wV \). 
We denote by \( \iiprd{\cdot}{\cdot} \) the standard
Euclidean inner product of \( V \) and
by \( \iinrm{\cdot} \) the standard
Euclidean norm of \( V \). 
Also, we use the symbols \( \iiprd{\cdot}{\cdot} \) and \(
\iinrm{\cdot} \) as the affine inner product
and affine norm of \( \wV \),
namely, for \( \wbs{v}_j =( n_j , \bs{v}_j )
\in \wV \; (j=1,2) \), we define \(
\iprd{\wbs{v}_1}{\wbs{v}_2} :=
n_1 n_2 + \iiprd{\bs{v}_1}{\bs{v}_2} \) and \(
\iiprd{\wbs{v}_1}{\wbs{v}_2} :=
\iiprd{\bs{v}_1}{\bs{v}_2} \).
We denote by \( \wV_0 := \{ \; \wbs{v} \in \wV \; | \;
\iinrm{\wbs{v}} =0 \; \} = \RR \oplus \{ \bs{0}_V \} \) the
isotropic subset of \( \wV \) and
by \( \wV_+  := \{ \; \wbs{v} \in \wV \; | \;
\iinrm{\wbs{v}} >0 \; \} = \wV \setminus \wV_0 \) its
complement.

Let \( \AVR \) be the set of all affine linear functions
from \( V \) to \( \RR \) and
let \( \AVV \) be the group of all affine linear transformations
of \( V \) with its multiplication defined by composition.
For any \( \wbs{v} \in \wV \),
the function \( f_{\wbs{v}} \) defined by \( f_{\wbs{v}} ( \bs{x} ):=
\iiprd{p_2 ( \wbs{v} )}{\bs{x}} + p_1 ( \wbs{v} ) \) belongs to \( \AVR \).
Since the map \( \iota : \wV \ni \wbs{v}
\stackrel{\sim}{\mapsto} f_{\wbs{v}} \in \AVR \) is
an isomorphism of \( \RR \)-vector spaces,
we identify \( \AVR \) with \( \wV \) in this section.
Let \( D:= p_2 \circ \iota^{-1} : \AVR \to V \).
We use the symbols \( \iiprd{\cdot}{\cdot} \) and \(
\iinrm{\cdot} \) also on \( \AVR \),
namely, we define \( \iiprd{f_1}{f_2} :=\iiprd{D f_1}{D f_2}
\) for \( f_1 , f_2 \in \AVR \) and \(
\iinrm{f} := \sqrt{\iiprd{f}{f}} \) for \( f \in \AVR \).
Let \( \AVR_+ := \iota ( \wV_+ )= \{ \; f \in \AVR \; | \;
\iinrm{f} >0 \; \} \) and
we denote by \( H_f := \{ \; \bs{x} \in V \; | \; f( \bs{x} )=0
\; \} \) be the set of the zeros of \( f \in \AVR_+ \),
which is a hyperplane in \( V \).
We denote by \( w_f : V \to V \) the reflection on \( V \) in
the hyperplane \( H_f \), which is defined by
\[
 w_f ( \bs{x} )
 := \bs{x} - \frac{2f( \bs{x} )}{\iinrm{f}^2} Df
 \qquad
 ( \bs{x} \in V)
 .
\]
Also we use the same symbol \( w_f \) as
the reflection on \( \AVR \), which is defined by
\[
 w_f (g) := g \circ w_f^{-1} = g \circ w_f
 = g - \frac{2 \iiprd{f}{g}}{\iinrm{f}^2} f
 \qquad
 (g \in \AVR)
 .
\]
For \( \wbsa \in \wV_+ \),
the reflection \( w_{\wbsa} : V \to V \) and \(
w_{\wbsa} : \wV \to \wV \) is
defined from \( w_{f_\wbsa} \) by the identification
of \( \wV \) with \( \AVR \), explicitly,
\begin{equation}
\label{eq:Arfv}
 w_{\wbsa} ( \bs{v} )
 := \bs{v} - \frac{2 f_{\wbsa} ( \bs{v} )}{\iinrm{\wbsa}^2} p_2 ( \wbsa )
 \qquad
 ( \bs{v} \in V)
\end{equation}
and
\begin{equation}
\label{eq:Arfa}
 w_{\wbsa} ( \wbs{v} )
 := \wbs{v} -
 \frac{2 \iiprd{\wbsa}{\wbs{v}}}{\iinrm{\wbsa}^2}
 \wbsa
 \qquad
 ( \wbs{v} \in \wV ).
\end{equation}
In the sense of (\ref{eq:Arfv}),
we have \( w_{\wbsa} \in \AVV \) and
this \( w_{\wbsa} \) is the reflection on \( V \) in
the hyperplane \( H_{\wbsa} := H_{f_{\wbsa}} \).
For \( \wR \subset \wV_+ \),
we denote by \( W( \wR ) \) the subgroup of \( \AVV \) generated by
reflections \( \{ \; w_{\wbsa} \; | \; \wbsa \in \wR \; \} \).
In the sense of (\ref{eq:Arfa}),
we have \( w_{\wbsa} \in \GL ( \wV ) \).
This induces a group homomorphism \(  W( \wR ) \to \GL ( \wV ) \)
and therefore we can define \( \det : W( \wR ) \to \RR \) from
\( \det : \GL ( \wV ) \to \RR \).
Since \( w_{\wbsa} ( \wbsa )=- \wbsa \) and \(
w_{\wbsa} ( \wbsb )= \wbsb \) if \(
\iiprd{\wbsa}{\wbsb} =0 \), we have \( \det w_{\wbsa} =-1 \).
We remark that
\begin{equation}
\label{eq:Afwi}
 f_{w( \wbsa )} ( w( \bs{x} )) = f_{\wbsa} ( \bs{x} )
 \qquad
 ( \; \wbsa \in \wV_+ , \; \bs{x} \in V , \; w \in W( \wR ) \;)
\end{equation}
holds, because
\[
 f_{w_{\wbsb} ( \wbsa )} ( w_{\wbsb} ( \bs{x} )) = f_{\wbsa} ( \bs{x} )
 \qquad
 ( \; \wbsa, \wbsb \in \wV_+ , \; \bs{x} \in V \; )
\]
holds.
\begin{defi}
We say \( \wR \subset \wV_+ \) is
an affine root system of rank \( N \) if
\( \wR \) satisfies the following four conditions:
\begin{description}
\item[(AR1)]
\( \wR \) spans \( \wV \) as an \( \RR \)-vector space.
\item[(AR2)]
For any \( \wbsa \in \wR \), \( w_{\wbsa} ( \wR )= \wR \).
(In the sense of (\ref{eq:Arfa}).)
\item[(AR3)]
For any \( \wbsa , \wbsb \in \wR \),
\( \dfrac{2 \iiprd{\wbsa}{\wbsb}}{\iinrm{\wbsa}^2} \in \ZZ \).
\item[(AR4)]
\( W( \wR ) \) acts properly on \( V \).
Namely, for any compact subsets \( K_1 , K_2 \subset V \),
the set \( \{ \; w \in W( \wR ) \; | \; w ( K_1 ) \cap K_2 \neq
\emptyset \; \} \) is finite.
(In the sense of (\ref{eq:Arfv}).)
\end{description}
\end{defi}
\begin{rema}
In the original definition given by Macdonald~\cite{Ma1},
affine root systems are subsets of \( \AVR \).
But in this section we give the definition of
affine root systems as subsets of \( \wV \),
by the identification of \( \AVR \) with \( \wV \).
This formulation is in the style of the theory
of extended affine root systems,
which was given by Saito~\cite{Sa1} about 10 years
after Macdonald~\cite{Ma1},
\end{rema}
\begin{rema}
\label{rema:A2Fp}
If \( \wR \) is an affine root system of rank \( N \),
then \( p_2 ( \wR ) \) is a finite root system of rank \( N \).
Hence \( p_2 ( \wR ) \) is a finite set
by {\bf{Proposition \ref{prop:Ffnt}}} and
thus \( \min ( \wR ) := \min \{ \; \iinrm{\wbsa} \; | \;
\wbsa \in \wR \; \} >0 \) exists.
\end{rema}
The condition {\bf{(AR4)}} is difficult to adapt to our purpose as it is,
so we rewrite it in a different form.
\begin{lemm}
\label{lemm:Afnc}
If \( \wR \) is an affine root system,
then the following property holds.
\begin{description}
\item[(AR4')]
The set \( \{ \; \wbsa \in \wR \; | \; | p_1 ( \wbsa ) | \leqq C \; \}
\) is finite for any \( C \in \RRP \).
\end{description}
Hence the set \( \wR \) is at most countable.
\end{lemm}
\begin{proof}
We denote by \( B_r := \{ \; \bs{x} \in V \; | \;
\iinrm{\bs{x}} \leqq r \; \} \) the solid closed ball
in \( V \) with its centre \( \bs{0}_V \) and its radius \( r>0 \).
By definition, \( w_{\wbsa} ( B_r ) \cap B_r \neq \emptyset \) if
and only if \( H_{\wbsa} \cap B_r \neq \emptyset \),
explicitly, \( \frac{| p_1 ( \wbsa )|}{\iinrm{\wbsa}}
\leqq r \).
The assumption {\bf{(AR4)}} implies that such \( \wbsa \in \wR \) is
finite for any fixed \( r>0 \).
By taking \( r= \frac{C}{\min ( \wR )} \),
we know that the condition {\bf{(AR4')}} holds.
\end{proof}
\begin{coro}
Let \( \wR \) be an affine root system.
For any \( \wbsn \in \wV \) satisfying \( p_1 ( \wbsn ) \neq 0 \),
the set \( \{ \; \wbs{s} \in \wR \; | \; | \iprd{\wbs{s}}{\wbsn} |
\leqq C \; \} \) is finite.
\end{coro}
\begin{proof}
It is easily shown
from {\bf{Remark \ref{rema:A2Fp}}} and {\bf{Lemma \ref{lemm:Afnc}}}.
\end{proof}
Now we assume that \( \wR \subset \wV_+ \) satisfies
the conditions {\bf{(AR2)}}, {\bf{(AR3)}} and {\bf{(AR4')}}.
Let
\begin{align*}
 R_1 & :=
 \{ \; \bsa \in p_2 ( \wR ) \; | \;
 \# ( p_2^{-1} ( \bsa ) \cap \wR )=1 \; \}
 ,
 &
 \wR_1 & := p_2^{-1} ( R_1 ) \cap \wR
 ,
 \\
 R_{\infty} & :=
 \{ \; \bsa \in p_2 ( \wR ) \; | \;
 \# ( p_2^{-1} ( \bsa ) \cap \wR )>1 \; \}
 &
 \text{and} \quad
 \wR_{\infty} & := p_2^{-1} ( R_{\infty} ) \cap \wR
 .
\end{align*}
Then we have \( p_2 ( \wR )= R_1 \sqcup R_{\infty} \) and
hence \( R_1 \) and \( R_{\infty} \) are finite sets
by {\bf{Remark \ref{rema:A2Fp}}}.
Because
\[
 p_2 ( w_{\wbsa} ( \wbsb ))= w_{ p_2 ( \wbsa) } ( p_2 ( \wbsb ))
 \qquad ( \; \wbsa \in \wV_+ , \; \wbsb \in \wV \;)
 ,
\]
where the right hand side is the reflection on \( V \),
and because \( w_{\wbsa} : \wR \to \wR \) is bijective,
we know that \( W( \wR ) \) acts on \( \wR_1 \) and on \( \wR_{\infty} \).
It is easy to see
that {\bf{(AR2)}}, {\bf{(AR3)}} and {\bf{(AR4')}} hold
even if we change \( \wR \) to \( \wR_1 \) or \( \wR_{\infty} \).

For each \( \bsa \in R_{\infty} \) we take \( q( \bsa ) \in \wR_{\infty}
\) satisfying \( p_2 (q( \bsa ))= \bsa \).
Because
\[
 w_{ \wbsa_2 } \circ w_{ \wbsa_1 } ( \wbsa )
 =
 \wbsa +2 ( \wbsa_2 - \wbsa_1 )
 \qquad
 ( \; \wbsa , \wbsa_1 , \wbsa_2 \in p_2^{-1} ( \bsa ) \; )
\]
holds,
we know the following:
\begin{itemize}
\item
\( p_2^{-1} ( \bsa ) \cap \wR \) is an infinite set.
\item
\( u_{\bsa} := \min \{ \; | p_1 ( \wbsa_1 - \wbsa_2 )| \; | \;
\wbsa_1 , \wbsa_2 \in p_2^{-1} ( \bsa ) \cap \wR , \;
\wbsa_1 \neq \wbsa_2 \; \} >0 \) exists by {\bf{(AR4')}}.
\item
\( p_2^{-1} ( \bsa ) \cap \wR =
 \{ \; q( \bsa )+k ( u_{\bsa} , \bs{0}_V ) \; | \; k \in \ZZ \; \} \).
\end{itemize}
\begin{prop}
\label{prop:Afnc}
Under the assumption {\bf{(AR1)}}, {\bf{(AR2)}} and {\bf{(AR3)}},
the condition {\bf{(AR4)}} is equivalent to the condition {\bf{(AR4')}}.
\end{prop}
\begin{proof}
We have already shown {\bf{(AR4)}}\( \Rightarrow \){\bf{(AR4')}}
in {\bf{Lemma \ref{lemm:Afnc}}}.
Here we show {\bf{(AR4')}}\( \Rightarrow \){\bf{(AR4)}}.
If \( R_{\infty} = \emptyset \), then \( W( \wR ) \) is a
finite group and
therefore {\bf{(AR4)}} holds automatically.
Hence we may assume \( R_{\infty} \neq \emptyset \).
Let \(  D_{k, \bsa } := \{ \; \bs{x} \in V \; | \;
| f_{q( \bsa )} ( \bs{x} )| < k u_{\bsa} \; \} \) for \(
k \in \NN \) and \( \bsa \in R_{\infty} \).
Then \( D_k := \bigcap_{\bsa \in R_{\infty}} D_{k, \bsa }
\) is an open set of \( V \).
Since \( \bigcup_{k \in \NN} D_k = V \), any compact set in \( V \) is
a subset of \( D_k \) for sufficiently large \( k \).
Hence it is enough to show that the set
\[
 W_k :=
 \{ \; w \in W( \wR ) \; | \; w ( D_k ) \cap D_k \neq \emptyset \; \}
\]
is finite for any \( k \in \NN \).
Let \( w \in W_k \) and \( \bsa \in R_{\infty} \).
We write \( \wbsb :=w(q( \bsa )) \) and \( \bsb := p_2 ( \wbsb) \).
Then we have \( \bsb \in R_{\infty} \) and \( w( D_{k, \bsa} )=
\{ \; \bs{x} \in V \; | \; | f_{\wbsb} ( \bs{x} )| < k u_{\bsa} \; \}
\) by (\ref{eq:Afwi}).
Because \(
D_{k, \bsb } \cap w( D_{k, \bsa } ) \neq \emptyset \) holds,
we have \( | p_1 ( \wbsb -q( \bsb ))|<k( u_{\bsa} + u_{\bsb} ) \).
We remark that the set
\[
 X_{\bsa} :=
 \{ \; \wbs{v}_{\bsa} \in \wR \; | \;
 p_2 ( \wbs{v}_{\bsa} ) \in R_{\infty} , \;
 | p_1 ( \wbs{v}_{\bsa} -q( p_2 ( \wbs{v}_{\bsa} )))|
 <k( u_{\bsa} + u_{ p_2 ( \wbs{v}_{\bsa} )} )
 \; \}
\]
is finite and that \( w(q( \bsa )) \in X_{\bsa} \) holds,
for each \( \bsa \in R_{\infty} \).
Since \( R_1 \) and \( R_{\infty} \) is finite, \(
w( R_1 )= R_1 \), \( w|_{ \wV_0 } \) is the identity map
and we assume the condition {\bf{(AR1)}}, the set \( W_k \) is finite.
\end{proof}

\bigskip

Here we prepare one elementary lemma,
which is used repeatedly in this section.
\begin{lemm}
\label{lemm:Atkn}
Let \( S \) be a nonempty subset of \( \wV \) which is
at most countable and let \( \wbs{p} \in \wV_+ \cup \{ \bs{0}_{\wV} \} \).
Then there exists a nonzero vector \( \wbsn \in \wV \) such that
\[
 \{ \; \wbs{s} \in S \; | \; \iprd{\wbs{s}}{\wbsn} =0 \; \}
 =
 \RR \wbs{p} \cap S
 \quad \text{and} \quad
 p_1 ( \wbsn )>0
 .
\]
\end{lemm}
\begin{proof}
A similar proof to {\bf{Lemma \ref{lemm:Ftkn}}} works.
\end{proof}
Let \( \wR \) be an affine root system.
We take a nonzero vector \( \wbsn \in \wV \) according
to {\bf{Lemma \ref{lemm:Atkn}}} with \( S= \wR \) and \( \wbs{p}
= \bs{0}_{\wV} \).
We denote by
\[
 \wR^+ := \{ \; \wbsa \in \wR \; | \; \iprd{\wbsa}{\wbsn} >0 \; \}
\]
the set of positive roots.
We remark that \( \wR= \wR^+ \sqcup (- \wR^+ ) \) holds,
where \( - \wR^+ := \{ \; - \wbsa \; | \; \wbsa \in \wR^+ \; \} \).
\begin{defi}
We say a root system \( \wR \) is reduced
if \( \wR \) satisfies the following condition:
\begin{description}
\item[(AR5)]
For any \( \wbsa \in \wR \), \(
\RR \wbsa \cap \wR = \{ \wbsa , - \wbsa \} \).
\end{description}
\end{defi}
Let \( \wR \) be a reduced affine root system.
For \( w \in W( \wR ) \), let
\[
 \wR (w) := \{ \; \wbsa \in \wR^+ \; | \; w( \wbsa ) \not\in \wR^+ \; \}
 \quad \text{and} \quad
 s(w) := \sum_{\wbsa \in \wR (w)} \wbsa
 ,
\]
where we define \( s(w)= \bs{0}_{\wV} \) if \( \wR (w)= \emptyset \).
We define the set of its positive imaginary roots by
\[
 \wR^+_{\mathrm{im}} := \{ \; k( u_{\bsa} , \bs{0}_V ) \in \wV_0
 \; | \; \bsa \in B( p_2 ( \wR )) , \; k \in \NN \; \}
\]
and let \( \mult ( \wbsa ) \) be
the multiplicity of \( \wbsa \in \wR^+_{\mathrm{im}} \).
Then, the following equation, which is called
the Weyl-Kac denominator formula or the Macdonald identity,
holds (cf.~\cites{Ka1,Ma1}):
\[
 \left( \prod_{\wbsa \in \wR^+_{\mathrm{im}}}
 \left( 1- e^{\wbsa} \right)^{\mult ( \wbsa )} \right)
 \left( \prod_{\wbsa \in \wR^+}
 \left( 1- e^{\wbsa} \right) \right)
 =
 \sum_{w \in W( \wR )} ( \det w) e^{s(w)}
 .
\]
\begin{defi}
We say a set \( \wR \subset \wV_+ \) is reducible
if there exist \( \wR_1 \) and \( \wR_2 \) such
that \( \wR = \wR_1 \sqcup \wR_2 \), \(
\wR_1 \neq \emptyset \), \( \wR_2 \neq \emptyset \) and \(
\iiprd{\wbsa}{\wbsb} =0 \) for
all \( \wbsa \in \wR_1 \) and for all \( \wbsb \in \wR_2 \).
We say a set \( \wR \) is irreducible
if it is not reducible.
\end{defi}
\subsection{Statement for the affine case}

Let \( m \in \map ( \wV, \NNI ) \).
Here we assume that \( m \) satisfies
the following properties {\bf{(A)}}:
\[
 \text{\bf{(A)}} \quad
 \left\{
 \begin{array}{l}
 \bullet \;
 m( \bs{0} )=0.
 \\[4pt]
 \bullet \;
 \text{The support \( S(m):= 
 \{ \; \wbs{v} \in \wV \; | \; m( \wbs{v} ) \neq 0 \; \} 
 \) is not empty.}
 \\[4pt]
 \bullet \;
 \text{For any \( \wbsn \in \wV \) satisfying \( p_1 ( \wbsn )>0 \) and
 for any \( C \in \RRP \),}\\
 \quad \text{\( \{ \; \wbs{s} \in S(m) \; | \;
 \iprd{\wbs{s}}{\wbsn} \leqq C \; \} \) is a finite set}.
 \\[4pt]
 \bullet \;
 \text{\( S(m) \) spans \( \wV \) as an \( \RR \)-vector space. }
 \\[4pt]
 \bullet \;
 \text{\( S(m)_{\mathrm{re}} :=
 \left\{ \; \wbs{s} \in S(m) \; \middle| \;
 \iinrm{\wbs{s}} >0 \; \right\} \) is irreducible.}
 \end{array}
 \right.
\]
We remark that the set \( S(m) \) is at most countable
by the third property of {\bf{(A)}}.
We define
\[
 F(m) := \prod_{\wbs{s} \in S(m)}
 \left( 1- e^{\wbs{s}} \right)^{m( \wbs{s} )}
\]
as a formal series.
We define \( \Lambda (m) \in \wV \) by
\[
 \Lambda (m) := \left\{ \; \wbsl \in \wV \; \middle| \;
 c_{\wbsl} \neq 0 \; \right\}
 \qquad \left(
 F(m)= \sum_{\wbsl} c_{\wbsl} e^{\wbsl}
 \right)
 .
\]

\bigskip

Our assertion in this section is as follows.
\begin{theo}
\label{theo:Aaim}
Under the assumption {\bf{(A)}},
the following two conditions on \( m \) are equivalent:
\begin{description}
\item[(A1)]
\( \Lambda (m) \) is a subset of a paraboloid.
Namely, there exist \( \wbs{c} \in \wV \) and \( r>0 \) such
that \( p_1 ( \wbsl - \wbs{c} )=r \iinrm{\wbsl - \wbs{c}}^2 \) holds for
any \( \wbsl \in \Lambda (m) \).
\item[(A2)]
\( S(m)_{\mathrm{re}} \cap (-S(m)_{\mathrm{re}} ) = \emptyset \), \( \wR :=
S(m)_{\mathrm{re}} \sqcup (-S(m)_{\mathrm{re}} ) \) is
an irreducible reduced affine root system
of rank \( N \) and \( m \) corresponds
to the index of Weyl-Kac denominator formula of
the affine root system, i.e.,
\[
 \left\{ \begin{array}{ll}
 m( \wbs{s} )= 1 & ( \; \wbs{s} \in S(m)_{\mathrm{re}} \; ) \\
 m( \wbs{s} )+ m(- \wbs{s} )= \mult ( \wbs{s} ) &
 ( \; \wbs{s} \in S(m) \setminus S(m)_{\mathrm{re}} \; )
 \end{array} \right.
 .
\]
\end{description}
\end{theo}

\bigskip

Here we prepare one lemma for later use.
\begin{lemm}
\label{lemm:Ar2s}
Let \( \wbsb \in S(m) \) and \( m^{\prime} \) be
a map which is defined by
\[
 m^{\prime} ( \wbs{v} ):=
 \left\{ \begin{array}{ll} m( \wbs{v} )-1 & ( \wbs{v} = \wbsb ) \\
 m( \wbs{v} )+1 & ( \wbs{v} =- \wbsb ) \\
 m( \wbs{v} ) & (\text{otherwise}) \end{array} \right.
 .
\]
Then we have \( \Lambda ( m^{\prime} ) = \Lambda (m) - \wbsb
:= \left\{ \; \wbsl - \wbsb \; \middle| \;
\wbsl \in \Lambda (m) \; \right\} \).
Thus the condition {\bf{(A1)}} holds for \( m^{\prime} \) if
and only if the condition {\bf{(A1)}} holds for \( m \).
\end{lemm}
\begin{proof}
A similar proof to {\bf{Lemma \ref{lemm:Fr2s}}} works.
\end{proof}
\subsection{Proof for the affine case}

The proof for {\bf{(A2)}}\( \Rightarrow \){\bf{(A1)}} is easy.
We take a nonzero vector \( \wbsn \in \wV \) according
to {\bf{Lemma \ref{lemm:Atkn}}} with \( S=S(m) \) and \( \wbs{p}
= \bs{0}_{\wV} \).
Just as in the case of finite root systems,
we may assume \( S(m) \subset \{ \; \wbs{v} \in \wV \; | \;
\iprd{\wbs{v}}{\wbsn} >0 \; \} \) and
\[
 \left\{ \begin{array}{ll}
 m( \wbs{s} )= 1 & ( \; \wbs{s} \in S(m)_{\mathrm{re}} = \wR^+ \; ) \\
 m( \wbs{s} )= \mult ( \wbs{s} ) &
 ( \; \wbs{s} \in S(m) \setminus S(m)_{\mathrm{re}} =
 \wR^+_{\mathrm{im}} \; )
 \end{array} \right.
 ,
\]
by {\bf{Lemma \ref{lemm:Ar2s}}}.
Then, \( F(m) \) is just the left-hand side
of the Weyl-Kac denominator formula of \( \wR \) and
hence \( \Lambda (m) \subset
\left\{ \; s(w) \; \middle| \; w \in W(R) \; \right\} \).
Therefore, because \( \wR \) is irreducible, \( \Lambda (m) \) is
a subset of a paraboloid by \cite{Ma1}*{Corollary 7.6}.
This completes the proof of {\bf{(F2)}}\( \Rightarrow \){\bf{(F1)}}.

\bigskip

We show {\bf{(A1)}}\( \Rightarrow \){\bf{(A2)}} step by step.
Throughout the proof, the key fact is that the intersection
of a paraboloid and a line has at most two points.

First, we show {\bf{(AR5)}}.
\begin{lemm}
\label{lemm:Ared}
Assume {\bf{(A1)}} and let \( \wbsa \in S(m)_{\mathrm{re}} \).
Then we have \( \RR \wbsa \cap S(m)_{\mathrm{re}} = \{ \wbsa \} \) and \(
m( \wbsa )=1 \).
\end{lemm}
\begin{proof}
The proof is very similar to that of {\bf{Lemma \ref{lemm:Fred}}}.
We take a nonzero vector \( \wbsn \in \wV \) according
to {\bf{Lemma \ref{lemm:Atkn}}} with \( S=S(m) \) and \(
\wbs{p} = \wbsa \).
By the third condition of {\bf{(A)}} and by {\bf{Lemma \ref{lemm:Ar2s}}},
we may assume \( \iprd{\wbs{s}}{\wbsn} \geqq 0 \) for
any \( \wbs{s} \in S(m) \) without loss of generality.
Then, because
\[
 S_0 :=
 \{ \; \wbs{s} \in S(m) \; | \; \iprd{\wbs{s}}{\wbsn} =0 \; \}
 = \RR \wbsa \cap S(m)
 ,
\]
we have
\[
 \Lambda_0 :=
 \{ \; \wbsl \in \Lambda (m) \; | \; \iprd{\wbsl}{\wbsn} =0 \; \}
 = \Lambda ( m |_{\RR \wbsa} )
 ,
\]
where \( m |_{\RR \wbsa} \in \map ( \RR \wbsa , \NNI )\) is
the restriction of \( m \) to
the subspace \( \RR \wbsa \subset \wV \).
Hence \( \Lambda_0 \) is on the line \( \RR \wbsa \).
On the other hand,
since \( \Lambda_0 \subset \Lambda (m) \), \( \Lambda_0 \) is
on a paraboloid, by the assumption {\bf{(A1)}}.
Therefore, \( \Lambda_0 \) is on the intersection
of the line and the paraboloid, which has at most two points.
However, since \( \wbsa \in S_0 \) and \( S_0 \) is
finite, \( \Lambda_0 \) has at least two points.
Consequently, \( \Lambda_0 \) has
exactly two points, which should be \( \bs{0}_{\wV} \) and \( \wbsa \).
This means \( \RR \wbsa \cap S(m)_{\mathrm{re}} = \{ \wbsa \}
\) and \( m( \wbsa )=1 \).
\end{proof}
Second, we show the {\bf{(AR3)}}.
\begin{lemm}
\label{lemm:Aint}
Assume {\bf{(A1)}} and let \( \wbsa \in S(m)_{\mathrm{re}} \).
Then we have \(
\frac{2 \iiprd{\wbsa}{\wbsb}}{\iinrm{\wbsa}^2}
\in \ZZ \) for any \( \wbsb \in S(m)_{\mathrm{re}} \).
\end{lemm}
\begin{proof}
The proof is similar to that of {\bf{Lemma \ref{lemm:Fint}}},
but since \( \Lambda (m) \) is a subset of a paraboloid here,
calculations are slightly different.
We take a nonzero vector \( \wbsn \in \wV \)
as in the previous proof.
Then we may assume \( \iprd{\wbsa}{\wbsn} =0 \) and \(
\iprd{\wbs{s}}{\wbsn} > 0 \) for
any \( \wbs{s} \in S(m) \setminus \{ \wbsa \} \).
According to the proof of {\bf{Lemma \ref{lemm:Ared}}},
we have \( \{ \bs{0}_{\wV} , \wbsa \} \subset \Lambda (m) \).
Therefore, \( \wbs{c} \in \wV \) and \( r>0 \) in
the assumption {\bf{(A1)}} satisfy the conditions \(
- p_1 ( \wbs{c} )=r \iinrm{\wbs{c}}^2 \) and \(
p_1 ( \wbsa )- p_1 ( \wbs{c} )=r \iinrm{\wbsa - \wbs{c}}^2 \).
Hence we have \(
p_1 ( \wbsa )=r( \iinrm{\wbsa}^2 -2 \iiprd{\wbsa}{\wbs{c}} ) \).

We show the assertion by contradiction.
Among the vectors \( \wbsb \in S(m)_{\mathrm{re}} \) satisfying \(
\frac{2 \iiprd{\wbsa}{\wbsb}}{\iinrm{\wbsa}^2} \not\in \ZZ \),
we take one with the minimal value of \( \iprd{\wbsb}{\wbsn} \).
Obviously \( \wbsb \neq \wbsa \),
hence \( \iprd{\wbsb}{\wbsn} >0 \).
Here, we may assume \( \wbsb - k \wbsa \not\in S(m) \) for any \( k>0 \).
Let \( L:= \wbsb + \RR \wbsa \),
which is a line and a subset of \( \{ \; \wbs{l} \in \wV \; | \;
\iprd{\wbs{l}}{\wbsn} = \iprd{\wbsb}{\wbsn} \; \} \).
To show the contradiction,
we investigate the set \( L \cap \Lambda (m) \) precisely.
By using similar argument in the proof
of {\bf{Lemma \ref{lemm:Fint}}}, \( L \cap \Lambda (m) \) has
exactly two points \( \wbsb \) and \( \wbsb +k \wbsa \),
where \( k>0 \) and \( k \in \ZZ \).
Since both \( \wbsb \) and \( \wbsb +k \wbsa \) are on \( \Lambda (m) \),
we have
\[
 p_1 ( \wbsb )- p_1 ( \wbs{c} )
 =
 r \iinrm{\wbsb - \wbs{c}}^2
 \quad \text{and} \quad
 p_1 ( \wbsb +k \wbsa )- p_1 ( \wbs{c} )
 =
 r \iinrm{\wbsb +k \wbsa - \wbs{c}}^2
 .
\]
Hence we have \( p_1 ( \wbsa )=r \left(
2 \iiprd{\wbsa}{ \wbsb - \wbs{c}} +k \iinrm{\wbsa}^2 \right) \).
Therefore, because \(
p_1 ( \wbsa )=r( \iinrm{\wbsa}^2 -2 \iiprd{\wbsa}{\wbs{c}} ) \),
we have \( \iinrm{\wbsa}^2 = 2 \iiprd{\wbsa}{\wbsb} +k
\iinrm{\wbsa}^2 \), namely, we have
\[
 \frac{2 \iiprd{\wbsa}{\wbsb}}{\iinrm{\wbsa}^2} =-k +1 \in \ZZ
 ,
\]
which is a contradiction.
\end{proof}
Third, we show {\bf{(AR2)}}.
\begin{lemm}
\label{lemm:Acls}
Assume {\bf{(A1)}} and let \( \wbsa \in S(m)_{\mathrm{re}} \).
Then we have \( w_{\wbsa} ( \wbsb ) \in S(m)_{\mathrm{re}} \sqcup
(-S(m)_{\mathrm{re}} ) \) for any \( \wbsb \in S(m)_{\mathrm{re}} \).
\end{lemm}
\begin{proof}
The proof is very similar to that of {\bf{Lemma \ref{lemm:Fcls}}}.
As in the previous proof,
we may assume \( \iprd{\wbsa}{\wbsn} =0 \) and \(
\iprd{\wbs{s}}{\wbsn} > 0 \) for
any \( \wbs{s} \in S(m) \setminus \{ \wbsa \} \).
Since \( w_{\wbsa} |_{\wV_0} \) is the identity transformation,
it is enough to show that the reflection \( w_{\wbsa} \) is a bijection
from \( S(m) \setminus \{ \wbsa \} \) to itself.
We can show it by contradiction,
by using similar argument in the proof
of {\bf{Lemma \ref{lemm:Fcls}}}.
\end{proof}

The condition {\bf{(AR1)}} is included in the assumption {\bf{(A)}} and
the condition {\bf{(AR4')}} easily follows from the assumption {\bf{(A)}}.
Therefore \( \wR :=S(m)_{\mathrm{re}} \sqcup (-S(m)_{\mathrm{re}} ) \) is
a reduced affine root system of rank \( N \) and it is irreducible
from the assumption {\bf{(A)}}.

\bigskip

Finally, we show the condition for \( m \).
Since we have already shown
\( m( \wbs{s} )= 1 \) for \( \wbs{s} \in S(m)_{\mathrm{re}} \),
all that remains to prove is the following lemma.
\begin{lemm}
\label{lemm:Acdm}
Assume {\bf{(A1)}}.
Then \( m( \wbs{s} )+ m(- \wbs{s} )= \mult ( \wbs{s} ) \) holds
for any \( \wbs{s} \in S(m) \setminus S(m)_{\mathrm{re}} \).
\end{lemm}
\begin{proof}
We take a nonzero vector \( \wbsn \in \wV \) according
to {\bf{Lemma \ref{lemm:Atkn}}} with \( S= \wR \) and \( \wbs{p}
= \bs{0}_{\wV} \).
We may assume \( S(m) \subset \{ \; \wbs{v} \in \wV \; | \;
\iprd{\wbs{v}}{\wbsn} >0 \; \} \) by
the third condition of {\bf{(A)}} and by {\bf{Lemma \ref{lemm:Ar2s}}}.
Then we may consider \( S(m)_{\mathrm{re}} = \wR^+ \) is
the set of all positive roots of \( \wR \) with respect to \( \wbsn \).
We define a map \( m_0 \in \map ( \wV, \NNI ) \) by
\[
 \left\{ \begin{array}{ll}
 m_0 ( \wbs{s} )= 1 & ( \; \wbs{s} \in \wR^+ \; ) \\
 m_0 ( \wbs{s} )= \mult ( \wbs{s} ) &
 ( \; \wbs{s} \in \wR^+_{\mathrm{im}} \; )
 \end{array} \right.
 .
\]
Namely, \( F( m_0 ) \) is the left-hand side of
the Weyl-Kac denominator formula of \( \wR \).
We show \( m= m_0 \).
Because \( m ( \wbs{s} )= m_0 ( \wbs{s} )=1 \) holds
for \( \wbs{s} \in S(m)_{\mathrm{re}} = \wR^+ \),
we have
\[
 F(m)
 \left( \prod_{\wbsa \in \wR^+_{\mathrm{im}}}
 \left( 1- e^{\wbsa} \right)^{\mult ( \wbsa )} \right)
 =
 F( m_0 )
 \left( \prod_{\wbs{s} \in S(m) \setminus S(m)_{\mathrm{re}}}
 \left( 1- e^{\wbs{s}} \right)^{m( \wbs{s} )} \right)
 .
\]
On the left-hand side, \( \Lambda (m) \) is a subset of a paraboloid
and \( \wR^+_{\mathrm{im}} \) is on the line \( \wV_0 \).
On the right-hand side, \( \Lambda ( m_0 ) \) is a subset of a paraboloid
and \( S(m) \setminus S(m)_{\mathrm{re}} \) is on the line \( \wV_0 \).
Therefore we have \( F(m)=F( m_0 ) \), \( \wR^+_{\mathrm{im}} =
S(m) \setminus S(m)_{\mathrm{re}} \) and \( m ( \wbs{s} )=
\mult ( \wbs{s} ) \) for any \( \wbs{s} \in \wR^+_{\mathrm{im}} \).
\end{proof}
Now we completes the proof of {\bf{(A1)}}\( \Rightarrow \){\bf{(A2)}}.

\section*{Acknowledgement}

The authors would like to thank
Professor Kyoji Saito for helpful discussions.
Also the authors are grateful to
Professors
Noriyuki Abe,
Alin Bostan,
Christopher H.~Chiu,
Jan Draisma,
Herwig Hauser,
Atsushi Matsuo,
Yoshihisa Saito
and Masahiko Yoshinaga
for useful advice.
This work is supported by
JSPS KAKENHI Grant Numbers 19K03429, 20K03546, 23K03039 and
24K06656 as well as by Nitto foundation.

\end{document}